\documentclass[reqno,12pt]{amsart}

\usepackage{amsmath,amsthm,amssymb,amsfonts,epsfig,color,graphicx,enumerate,stackrel}
\usepackage[a4paper,margin=3.25truecm]{geometry}
\usepackage[T1]{fontenc}

\DeclareMathOperator{\dive}{div}

\DeclareMathOperator{\dom}{dom}

\def\ds{\displaystyle}
\def\eps{{\varepsilon}}

\def\R{\mathbb{R}}

\def\O{\Omega}
\def\A{\mathcal{A}}

\def\E{\mathcal{E}}
\def\F{\mathcal{F}}

\def\M{\mathcal{M}}

\def\RR{\mathcal{R}}

\def\V{\mathcal{V}}

\newcommand{\be}{\begin{equation}}
\newcommand{\ee}{\end{equation}}
\newcommand{\bib}[4]{\bibitem{#1}{\sc#2: }{\it#3. }{#4.}}

\numberwithin{equation}{section}
\theoremstyle{plain}
\newtheorem{theo}{Theorem}[section]

\theoremstyle{remark}
\newtheorem{rema}[theo]{\bf Remark}
\newtheorem{exam}[theo]{Example}

\title[Optimization problems for elliptic PDEs]{Optimization problems for elliptic PDEs}

\author{\centerline{G. Buttazzo,\hskip 0.3cm J. Casado-D\'{\i}az,\hskip 0.3cm F. Maestre}}

\date{\today}

\begin{document}

\maketitle

\hfill{\it Dedicated to Gioconda Moscariello for her 70th birthday}

\begin{abstract}
In this paper we consider some optimal control problems governed by elliptic partial differential equations. The solution is the state variable, while the control variable is, depending on the case, the coefficient of the PDE, the potential, the right-hand side. The cost functional is of integral type and involves both the state and control variables.
\end{abstract}

\bigskip



\section{Introduction}\label{sintro}

In this paper we consider optimal control problems governed by elliptic partial differential equations; the state variable is the corresponding solution, while for the control variable we analyze the three different situations below. In all the cases the cost functional is of integral type and may depend both on the state and the control variables. The boundary condition is taken of the Dirichlet type; making suitable arrangements in the framework, other types of boundary conditions could be considered.

We summarize here in a unified framework the results we obtained in a series of papers (see \cite{BCM24}, \cite{BCM25}, \cite{BCM26}), referring to them for all the proofs and details.
\begin{itemize}
\item The first case we consider is when the control is in the coefficient of the PDE; more precisely, the state equation is of the form
\be\label{pde1}
\begin{cases}
-\dive\big(a(x)\nabla u\big)=f&\text{in }\O\\
u=0&\text{on }\partial\O.
\end{cases}\ee
Here $\O$ is a given bounded open subset of $\R^d$ and the right-hand side $f$ is prescribed; the control variable is the coefficient $a$. The set of admissible controls is taken of the form
\be\label{contr1}
\A=\bigg\{\int_\O\psi(a)\,dx\le1\bigg\}
\ee
and the cost functional is
$$J(u,a)=\int_\O j(u,a)\,dx.$$
The optimal control problem is then
\be\label{pb1}
\min\big\{J(u,a)\ :\ (u,a)\text{ satisfy }\eqref{pde1},\ a\in\A\big\}.
\ee
In Section \ref{scoeff} we summarize the known result in this case, mainly obtained in \cite{BCM25}, where the reader may also find the related references.

\item In the second case, that we illustrate in Section \ref{spoten}, the control variable is a potential $V(x)$ and the PDE is of the form
\be\label{pde2}
\begin{cases}
-\Delta u+V(x)u=f&\text{in }\O\\
u=0&\text{on }\partial\O.
\end{cases}\ee
In this case the set of admissible controls is of the form
$$\V=\bigg\{\int_\O\psi(V)\,dx\le1\bigg\}$$
and the cost functional is
$$J(u,V)=\int_\O j(u,V)\,dx.$$
The optimal control problem, in this second case, is then
$$\min\big\{J(u,V)\ :\ (u,V)\text{ satisfy }\eqref{pde2},\ V\in\V\big\}.$$

\item The last situation, that we consider in Section \ref{srhs}, is when the control variable is the right-hand side $f$. The PDE is of the form
\be\label{pde3}
\begin{cases}
-\Delta u=f&\text{in }\O\\
u=0&\text{on }\partial\O
\end{cases}\ee
and the set of admissible controls is
$$\F=\bigg\{\int_\O\psi(f)\,dx\le1\bigg\}.$$
Therefore, in this case the optimal control problem we deal with is
$$\min\big\{J(u,f)\ :\ (u,f)\text{ satisfy }\eqref{pde3},\ f\in\F\big\}.$$
\end{itemize}

\section{Optimal coefficients}\label{scoeff}

In this section we consider the optimization problem \eqref{pb1}.

\subsection{The case of optimal compliance}\label{compliance1}

We start by the particularly interesting problem of {\it minimal compliance}, where the cost functional to be minimized is
$$C(a)=\int_\O f(x)u_a\,dx,$$
being $u_a$ the solution of the PDE \eqref{pde1}. The minimal compliance problem is crucial in the construction of optimized mechanical structures, where one wants to utilize a predetermined quantity of material to achieve maximum rigidity under a specified force field.

A simpler way to impose the constraint on $a$ is to write the problem in the form
\be\label{lagr1}
\min\bigg\{C(a)+\lambda\int_\O\psi(a)\,dx\ :\ a\ge0\bigg\},
\ee
where $\lambda>0$ plays the role of a Lagrange multiplier and $\psi$ is a convex lower semicontinuous function. Replacing $\lambda\psi$ by $\psi$ we can also assume $\lambda=1$.

Assuming $a$ in $L^\infty(\Omega)$ and uniformly elliptic to assure the existence of solution for (\ref{pde1}), multiplying both sides of \eqref{pde1} by $u$ and integrating by parts we obtain that
$$C(a)=-2E(a),$$
where $E(a)$ is the energy
$$E(a)=\inf\bigg\{\int_\O\Big[\frac12 a(x)|\nabla u|^2-f(x)u\Big]\,dx\ :\ u\in H^1_0(\O)\bigg\}.$$
Therefore the optimal control problem \eqref{lagr1} is equivalent to the maximization problem
$$\max\bigg\{2E(a)-\int_\O\psi(a)\,dx\ :\ a\ge0\bigg\},$$
and then it can be written in the max/min form
$$\max_{a\ge0}\inf_{u\in H^1_0(\O)}\int_\O\Big[a(x)|\nabla u|^2-2f(x)u-\psi(a)\Big]\,dx.$$
In some cases we have to deal with right-hand sides $f$ which are singular (for instance concentrated forces, \dots); it is then convenient to assume that $f$ is a signed measure, replacing the Sobolev space $H^1_0(\O)$ with $C^\infty_c(\O)$. It must be noticed that we may have $C(a)=+\infty$ for some coefficients $a$; this happens for instance in the case when the right-hand side $f$ concentrates on sets of dimension smaller than $d-1$ and the coefficient $a$ is a positive constant. However, these ``singular'' cases are ruled out from our discussion because we look for the minimization of the compliance $C(a)$, hence the coefficients $a$ for which $C(a)=+\infty$ are not relevant for the optimization problem. The energy $E(a)$ then takes the form
$$E(a)=\inf\bigg\{\int_\O\Big[\frac12 a(x)|\nabla u|^2\Big]\,dx-\int_\O u\,df\ :\ u\in C^\infty_c(\O)\bigg\}$$
and the max/min problem above becomes
\be\label{maxmin1}
\max_{a\ge0}\inf_{u\in C^\infty_c(\O)}\left\{\int_\O\Big[a(x)|\nabla u|^2-\psi(a)\Big]\,dx-2\int_\O u\,df\right\}.
\ee
In this formulation, since the competing functions $u$ are smooth, also the coefficient $a$ can be a (nonnegative) measure; in particular, this is important when the penalization function $\psi$ has a linear growth. In this way thin structures supported on lower dimensional sets are admissible. For a nonnegative measure $\mu$ the energy functional takes the form
$$E(\mu)=\inf\bigg\{\int_\O\frac12|\nabla u|^2\,d\mu-\int_\O u\,df\ :\ u\in C^\infty_c(\O)\bigg\}$$
and the max/min problem becomes
$$\max_{\mu\ge0}\inf_{u\in C^\infty_c(\O)}\int_\O|\nabla u|^2d\mu-\int\psi(\mu)-2\int_\O u\,df.$$
The term $\int\psi(\mu)$ has to be intended in the sense of integral functionals on measures, as
\be\label{fmu}
\int_\O\psi\big(\mu^a(x)\big)\,dx+\psi^\infty(1)\,\mu^s(\overline\O)
\ee
where $\mu=\mu^a\,dx+\mu^s$ is the decomposition of $\mu$ into absolute continuous part and singular part with respect to the Lebesgue measure and $\psi^\infty$ is the {\it recession function} associated to $\psi$.

We start to consider first the case when the function $\psi$ in \eqref{contr1} is superlinear, that is
\be\label{superl1}
\lim_{s\to+\infty}\frac{\psi(s)}{s}=+\infty;
\ee
in this case it is well-known that the coefficients $a$ have to be functions in $L^1(\O)$.

\begin{theo}\label{exL1}
Under the superlinearity assumption \eqref{superl1}, the max/min problem \eqref{maxmin1} admits a solution $a_{opt}\in L^1(\O)$, provided the right-hand side $f$ is such that $E(a)>-\infty$ for at least a coefficient $a\in L^1(\O)$.
\end{theo}

\begin{proof}
We refer to \cite{BCM25} for the proof. It is interesting to notice that the functional
$$a\mapsto\inf_{u\in C^\infty_c(\O)}\int_\O\Big[a(x)|\nabla u|^2-\psi(a)\Big]\,dx-2\int_\O u\,df$$
is concave and upper semicontinuous for the weak $L^1(\O)$ convergence, being the infimum of concave and upper semicontinuous functionals. The compactness property comes from the well-known De La Vall\'ee Poussin theorem.
\end{proof}

The case when $\psi$ has only a linear growth:
\be\label{ling}
\lim_{s\to+\infty}\frac{\psi(s)}{s}=k>0
\ee
is more delicate. In fact, in this case the definition \eqref{fmu} of $\int_\O\psi(\mu)$ is needed. We refer to \cite{BGL} for more details about this case, which has strong links with the theory of optimal transportation, as first shown in \cite {BBS} and \cite{BB}. However, by an argument similar to the one of Theorem \ref{exL1}, an optimal coefficient $a_{opt}$ still exists, but in the larger class $\M(\O)$ of nonnegative measures on $\O$, as stated below (we refer to \cite{BCM25} for the proof, similar to the one of Theorem \ref{exL1} above).

\begin{theo}\label{exM}
Under the linear growth assumption \eqref{ling}, the functional $E(a)$ admits a maximizer $a_{opt}$ in the class $\M(\O)$, provided the right-hand side $f$ is such that $E(a)>-\infty$ for at least a coefficient $a\in\M(\O)$.
\end{theo}

Our goal is now to characterize the optimal coefficient $a_{opt}$ in terms of some suitable auxiliary variational problem. Due to the convexity with respect to the variable $u$ and the concavity with respect to the variable $a$, a well-known result from min/max theory (see for instance \cite{Cl} and \cite{Ek}) allows us to exchange the order of inf and sup, and the initial problem becomes

$$\inf_{u\in C^\infty_c(\O)}\sup_{a\ge0}\int_\O\Big(a(x)|\nabla u|^2-\psi(a)\Big)\,dx-2\int u\,df.$$
The supremum with respect to $a$ can be now easily computed:
$$\sup_{a\ge0}\int_\O\Big(a(x)|\nabla u|^2-\psi(a)\Big)\,dx-2\int u\,df=\int_\O\psi^*\big(|\nabla u|^2\big)\,dx-2\int u\,df,$$
where $\psi^*$ is the Legendre-Fenchel conjugate function of $\psi$. The auxiliary variational problem is then
$$\inf_{u\in C^\infty_c(\O)}\int_\O\psi^*\big(|\nabla u|^2\big)\,dx-2\int u\,df.$$
By the definition of Legendre-Fenchel conjugate we obtain
$$\psi^*(t)\ge t-\psi(1)\qquad\text{for every }t$$
and then it is easy to see that, at least when $f\in H^{-1}(\O)$, the auxiliary variational problem above admits a solution $\bar u\in H^1_0(\O)$. Moreover, if $\psi$ is strictly increasing on $\R^+$, then the function $s\mapsto\psi^*(s^2)$ is strictly convex and therefore $\bar u$ is unique. The optimal coefficient $a_{opt}$ can now be recovered through the optimality condition
$$a_{opt}|\nabla\bar u|^2=\psi(a_{opt})+\psi^*(|\nabla\bar u|^2).$$

\begin{exam}\label{Ex1}
Let $\psi(s)=s^p/p$ with $p>1$ and assume $f\in W^{-1,2p/(p+1)}(\O)$. By Theorem \ref{exL1} there exists a (unique) optimal coefficient $a_{opt}\in L^1(\O)$, which is indeed in $L^p(\O)$. It can be recovered through the auxiliary variational problem
$$\inf\bigg\{\int_\O\frac{p-1}{p}|\nabla u|^{2p/(p-1)}\,dx-2\int u\,df\ :\ u\in W^{1,2p/(p-1)}_0(\O)\bigg\}.$$
or equivalently by the nonlinear PDE
$$-\Delta_{2p/(p-1)}u=\frac{2p}{p-1}f,\qquad u\in W^{1,2p/(p-1)}_0(\O),$$
whose unique solution $\bar u$ provides the optimal coefficient $a_{opt}(x)=|\nabla\bar u(x)|^{2/(p-1)}$. For instance, if $\O$ is the unit ball, and $f=1$ we obtain (see Figure \ref{figura1})
$$\bar u(x)=\frac{p+1}{2p\,d^{(p-1)/(p+1)}}\big(1-|x|^{2p/(p+1)}\big),\qquad a_{opt}(x)=\frac{|x|^{2/(p+1)}}{d^{2/(p+1)}}.$$

\begin{figure}[h!]
\centering
{\includegraphics[scale=0.5]{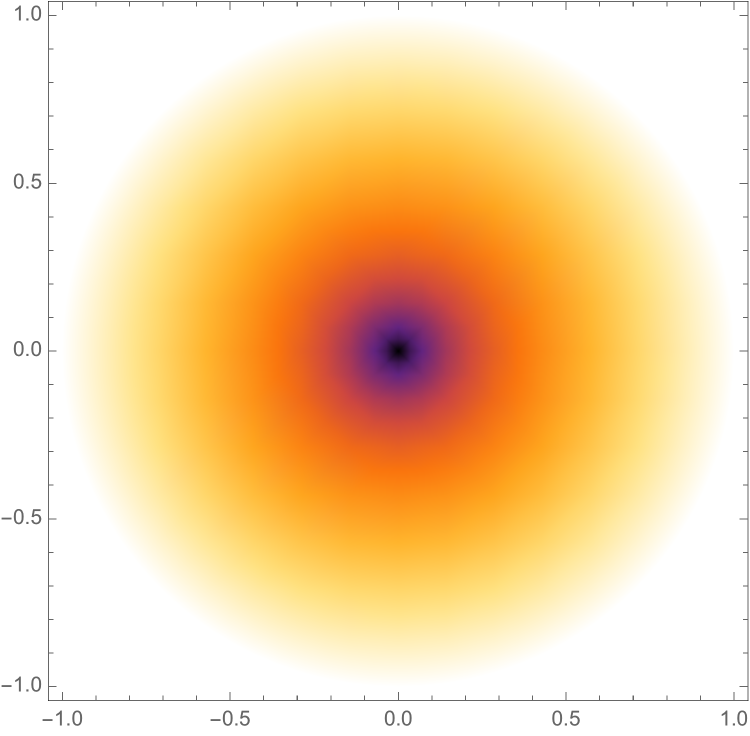}}
\caption{Optimal coefficient when $d=2$, $\O$ is the unit disk, $p=2$, $f=1$.}\label{figura1}
\end{figure}

When $p<d/(d-2)$ measures belong to $W^{-1,2p/(p+1)}(\O)$; taking $f=\delta_0$ the unit Dirac mass at the origin and $\O$ the unit ball, we obtain
$$\bar u(x)=A_p(1-|x|^{(2p-dp+d)/(p+1)}),\qquad a_{opt}(x)=B_p|x|^{-2(d-1)/(p+1)}$$
with $A_p,B_p$ suitable positive constants that can be computed explicitly. For instance, in the two-dimensional case $d=2$, we find (see Figure \ref{figura2})
$$A_p=\frac{p+1}{2}(2\pi)^{(1-p)/(1+p)},\qquad B_p=(2\pi)^{-2/(1+p)}.$$

\begin{figure}[h!]
\centering
{\includegraphics[scale=0.5]{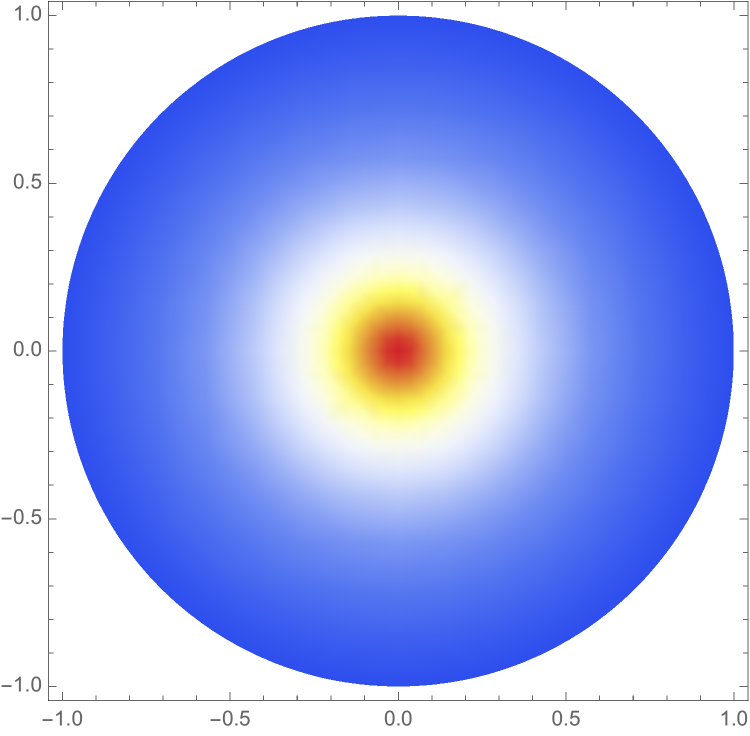}}
\caption{Optimal coefficient when $d=2$, $\O$ is the unit disk, $p=2$, $f=\delta_0$.}\label{figura2}
\end{figure}
\end{exam}

\begin{exam}\label{Ex2}
Taking
$$\psi(s)=\begin{cases}
s&\hbox{if }\alpha\le s\le\beta\\
+\infty&\hbox{otherwise,}
\end{cases}$$
with $0<\alpha<\beta$, we have the auxiliary variational problem
$$\min\bigg\{\int_\O\Big[\big(|\nabla u|^2-1\big)\big(\beta1_{\{|\nabla u|>1\}}+\alpha1_{\{|\nabla u|<1\}}\big)-2f(x)u\Big]\,dx\ :\ u\in H^1_0(\O)\bigg\},$$
whose unique solution $\bar u$ provides the optimal coefficient $a_{opt}\in L^\infty(\O)$. It has been proved in \cite{Cas} (see also \cite{Cas2}) that, when $\O$ is of class $C^{1,1}$ and $f\in L^2(\O)$, then $\bar u$ is in $H^2(\O)$ and $\nabla a_{opt}\cdot\nabla \bar u$ belongs to $L^2(\O)$.
\end{exam}

\subsection{The general case}\label{control1}

In the case of a general optimal control problem of the form
\be\label{pbge}
\min_{a\ge0}\min_{u\in H^1_0(\O)}\Big\{\int_\O\big(j(x,u)+\psi(a)\big)\,dx\ :\ u\hbox{ solves \eqref{pde1}}\Big\},
\ee
the existence of an optimal coefficient $a_{opt}$ may fail, and a solution exists only in a {\it relaxed} sense. A counterexample to the existence of a solution $a_{opt}$ can be found in \cite{Mur}, where
$$\psi(s)=\begin{cases}
0&\hbox{if }\alpha\le s\le\beta\\
+\infty&\hbox{otherwise,}
\end{cases}$$
with $0<\alpha<\beta$, and
$$j(x,s)=|s-u_0(x)|^2$$
for a suitable function $u_0$. A different counterexample is illustrated in \cite{BCM25}.

The reason of the lack of existence of an optimal coefficient is the form of the relaxed problem associated to \eqref{pbge}. This question has been largely studied, in terms of the notion of $G$-convergence, introduced by De Giorgi and Spagnolo in \cite{DGS}: a sequence $a_n(x)$ of functions between $\alpha$ and $\beta$ is said to $G$-converge to a symmetric $d\times d$ matrix $A(x)$ if for every $f\in L^2(\O)$ the solutions $u_n$ of the PDEs
$$-\dive\big(a_n\nabla u_n\big)=f,\qquad u_n\in H^1_0(\O)$$
converge in $L^2(\O)$ to the solution $u$ of the PDE
$$-\dive\big(A\nabla u\big)=f,\qquad u\in H^1_0(\O).$$

The relaxed form of the optimal control problem \eqref{pbge} is then closely related to the characterization of the $G$-closure $\overline\A$ of the set $\A$ of admissible coefficients $a(x)$. A complete answer has been given by Murat and Tartar in \cite{MT}, \cite{Tar} (see also \cite{All}, and \cite{LC} for the two-dimensional case). They proved that the $G$-closure $\overline\A$ above consists of all symmetric $d\times d$ matrices $A(x)$ whose eigenvalues $\lambda_1(x)\le\lambda_2(x)\le\dots\le\lambda_d(x)$ are between $\alpha$ and $\beta$ and satisfy for a suitable $t\in[0,1]$ (depending on $x$) the following $d+2$ inequalities:
\be\label{relset}
\begin{cases}
\ds\sum_{1\le i\le d}\frac{1}{\lambda_i-\alpha}\le\frac{1}{\nu_t-\alpha}+\frac{d-1}{\mu_t-\alpha}\\
\ds\sum_{1\le i\le d}\frac{1}{\beta-\lambda_i}\le\frac{1}{\beta-\nu_t}+\frac{d-1}{\beta-\mu_t}\\
\nu_t\le\lambda_i\le\mu_t\qquad i=1,\dots,d,
\end{cases}
\ee
being $\mu_t$ and $\nu_t$ respectively the arithmetic and the harmonic mean of $\alpha$ and $\beta$, namely
$$\mu_t=t\alpha+(1-t)\beta,\qquad\nu_t=\Big(\frac{t}{\alpha}+\frac{1-t}{\beta}\Big)^{-1}.$$
For instance, when $d=2$, the set above is given by the symmetric $2\times2$ matrices $A(x)$ whose eigenvalues $\lambda_1(x)$ and $\lambda_2(x)$ are between $\alpha$ and $\beta$ and satisfy the inequality
$$\frac{\alpha\beta}{\alpha+\beta-\lambda_1(x)}\le\lambda_2(x)\le\alpha+\beta-\frac{\alpha\beta}{\lambda_1(x)}\;.$$
In Figure \ref{lens} we represent the case $d=2$, $\alpha=1$, $\beta=2$.

\begin{figure}[h!]
\centering
{\includegraphics[scale=1.0]{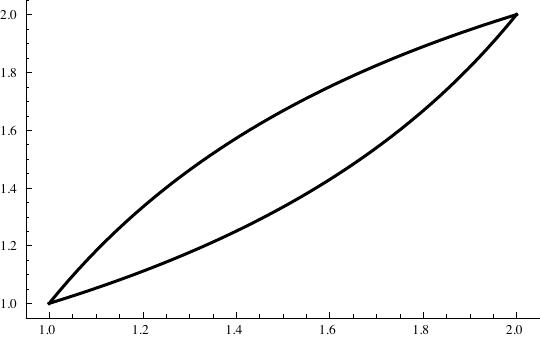}}
\caption{Attainable matrices, in the plane $(\lambda_1,\lambda_2)$, for $d=2$, $\alpha=1$, $\beta=2$.}\label{lens}
\end{figure}

\begin{rema}
For a general function $\psi(x,a)$, an explicit form of the relaxation
$$\Psi(A)=\inf_{a_n\to_G A}\liminf_n\int_\O\psi(x,a_n)\,dx$$
is not known. The case $\psi(x,a)=g(x)a$ has been considered in \cite{Cab} and \cite{CDM}; for instance, in the particular case when $g(x)$ is a constant $\gamma$ and the coefficient $a$ is constrained between two positive constants $\alpha$ and $\beta$, denoting by $\lambda_{max}(A)$ the largest eigenvalue of the $d\times d$ symmetric matrix $A$, the relaxation $\Psi(A)$ is given by
$$\Psi(A)=\begin{cases}\ds
\int_\O\gamma\lambda_{max}\big(A(x)\big)\,dx&\text{if }A\in\overline\A\\
+\infty&\text{otherwise,}
\end{cases}$$
where $\overline\A$ is the $G$-closure described above.
\end{rema}

\begin{exam}\label{Ex22}
For $\Omega=(0,1)^2$ we take
$$j(x,s)=s,\quad\quad f(x_1,x_2)=x_1^2,$$
$$\psi(s)=\begin{cases}
s&\hbox{if }5\le s\le 10\\
+\infty&\hbox{otherwise,}
\end{cases}$$
we solve numerically the relaxed formulation of the problem searching an optimal solution $A_{opt}$ whose eigenvalues satisfy (\ref{relset}).

\begin{figure}[!t]
\begin{minipage}[!h]{19cm}
\centering
\hspace{-3.5cm}\includegraphics[width=9.5cm]{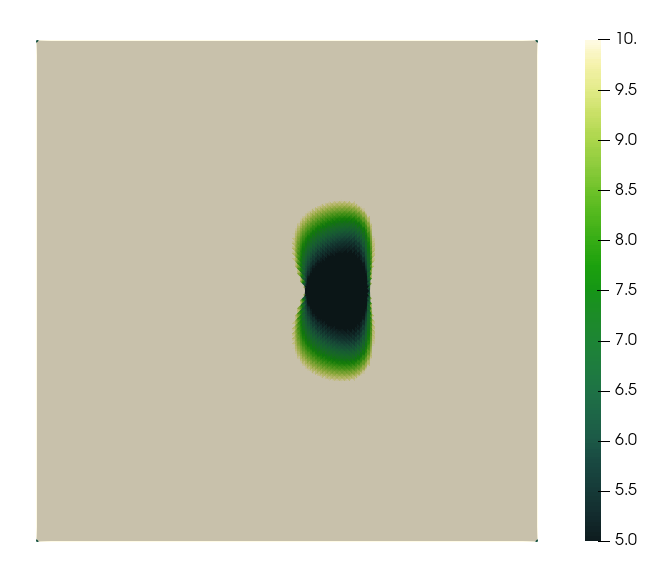}\hspace{-.2cm}
\includegraphics[width=9.5cm]{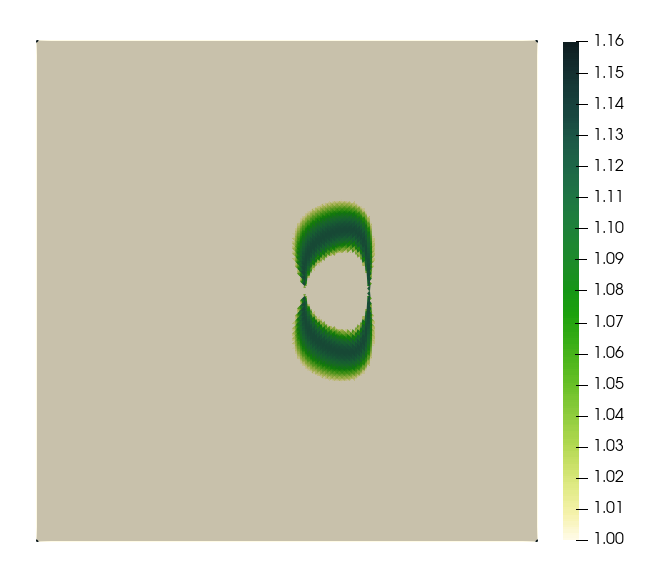}
\end{minipage}
\caption{Left: optimal first eigenvalue. Right: ratio $\lambda_2/\lambda_1$.} \label{eigenvalues}
\end{figure}

In Figure \ref{eigenvalues} we show, on the left, the computed optimal first eigenvalue of the optimal matrix $A_{opt}$, and on the right, the ratio $\lambda_2/\lambda_1$. It is interesting to observe the existence of zones with positive measure where $\lambda_1\not=\lambda_2$ and then that $A_{opt}$ is a relaxed solution.
\end{exam}

\section{Optimal potentials}\label{spoten}

In this section, we investigate a class of optimization problems characterized by the minimization of functionals of the form
\be\label{poten}
\min\int_\O\big[j(x,u)+\psi(V)\big]\,dx,
\ee
governed by the state equation
\be\label{pstate}
\begin{cases}
-\Delta u+Vu=f\quad\text{in }\O\\
u\in H^1_0(\O),
\end{cases}
\ee
where, as in the previous section, $\O$ is a bounded open subset of $\R^d$, and the right-hand side $f\in L^2(\O)$. The integrand $j(x,u)$ represents the contribution of the state variable $u$, while $\psi(V)$ models the cost associated with the control variable $V$, which in this context is the potential function and is assumed to be nonnegative almost everywhere in $\O$.

We further assume that $\psi$ has a superlinear growth at infinity. This growth condition ensures, by virtue of the De La Vall\'ee Poussin theorem, that any admissible potential $V$ of finite cost must belong to the space $L^1(\O)$.

Again, the compliance case, corresponding to the choice $j(x,u)=f(x)u$, is particularly simple since, integrating by parts equation \eqref{pstate}, we may reduce the problem to the variational form as
\be\label{poptpb}
\min\Big\{-2\E(V)+\Psi(V)\ :\ V\in L^1(\O),\ V\ge0\Big\},
\ee
where the functionals $\E$ and $\Psi$ are defined by
\[\begin{split}
&\Psi(V)=\int_\O\psi(V)\,dx\\
&\E(V)=\min\left\{\int_\O\Big[\frac12|\nabla u|^2+\frac12 Vu^2-f(x)u\Big]\,dx\ :\ u\in H^1_0(\O)\right\}.
\end{split}\]

In this specific case, one can exploit the variational structure of the problem to eliminate the control variable $V$ entirely, reducing the original optimal control problem to a simpler auxiliary variational problem. More precisely, it is possible to reformulate the minimization \eqref{poptpb} in terms of the state variable $u$ only, as
\be\label{aux}
\min\left\{\int_\O\left[|\nabla u|^2+\psi^*(u^2)-2f(x)u\right]\,dx\ :\ u\in H^1_0(\O)\right\},
\ee
where $\psi^*$ denotes the Legendre-Fenchel transform of the function $\psi$. We observe that the unique minimizer $\bar u$ of the variational problem \eqref{aux} can be characterized as the solution of the semilinear elliptic boundary value problem
$$\begin{cases}
-\Delta u+g(u)=f&\text{in }\O,\\
u\in H^1_0(\O),
\end{cases}$$
where the function $g$ is defined as
$$g(s)=s(\psi^*)'(s^2).$$
The optimal control $V_{opt}$ can now be recovered in terms of the optimal state $\bar u$ via the formula
$$V_{opt}=(\psi^*)'(\bar u^2).$$
It is important to note, however, that this reduction technique, based on the elimination of the control variable, is specific to the case where the integrand $j(x,u)$ in the cost functional is the one related to the compliance case. In the general setting, where $j(x,u)$ is an arbitrary function of the state, such elimination is no longer feasible. Instead, the analysis must rely on the optimality system, which typically involves introducing an adjoint state variable and solving the corresponding adjoint PDE, coupled with the original state equation.

Despite these complexities, it can be shown that the original optimal control problem \eqref{poten} with state equation \eqref{pstate}, admits at least one solution pair $(\bar u,V_{opt})$. This result has been obtained in \cite{BCM24}, and we report the existence theorem below.

\begin{theo}\label{expot}
Let $\O\subset\R^d$ be a bounded open set, $j:\O\times\R\to\R$ an integrand measurable in the first component and lower semicontinuous in the second one, satisfying the growth condition
$$a(x)-c|s|^2\le j(x,s)$$
for suitable $c\ge0$ and $a\in L^1(\O)$, and $\psi:\R\to[0,\infty]$ a convex lower semicontinuous function with $\dom(\psi)\subset[0,\infty)$, such that the superlinearity condition 
$$\lim_{|s|\to\infty}\frac{\psi(s)}{s}=+\infty$$
holds. Then, for every $f\in H^{-1}(\O)$, problem \eqref{poten} has a least one solution $V_{opt}\in L^1(\O)$.
\end{theo}

In \cite{BCM24} the properties of the solutions have been studied too. In particular, under appropriate assumptions on the data and the integrands, the quantity $V_{opt}\bar u v$ is shown to belong to the space of functions of bounded variation, $BV(\O)$, where $v$ denotes the solution of the adjoint equation (\ref{padj}). This implies the important consequence that the optimal control $V_{opt}$ is locally $BV$ outside of the set $\{\bar u v=0\}$. Thus, the degeneracy set $\{\bar u v=0\}$ plays a central role in determining the structure of singularities in the optimal control.

To be more precise, in \cite{BCM24} the following necessary conditions of optimality have been obtained.

\begin{theo}\label{TConOp}
Assume that the right-hand side $f$ is in $W^{-1,r}(\O)$ with $r>d$, and that the integrand $j(x,\cdot)$ is of class $C^1(\R)$ and satisfies
$$j(\cdot,0)\in L^1(\O),\qquad\max_{|s|\le M}|\partial_s j(\cdot,s)|\in L^{r/2}(\O)\quad\forall\,M>0.$$
Then, if $V_{opt}$ is an optimal control for problem \eqref{poten}, $\bar u$ is the corresponding state function, solution of \eqref{pstate}, and $v$ is the adjoint state, solution of
\be\label{padj}
\begin{cases}
-\Delta v+V_{opt}v=\partial_sj(x,\bar u)&\text{in }\O\\
v=0&\text{on }\partial\O,
\end{cases}\ee
we have the optimality conditions
\be\label{pcondopt}
\begin{cases}
V_{opt}\in L^\infty(\O)\\
\bar u v\in\partial\psi(V_{opt})\\
h_-(\bar u v)\le V_{opt}\le h(\bar u v),
\end{cases}\ee
where $h:\R\to\R$ is the function
\be\label{functionh}
h(t)=\max\big\{s\in\dom(\psi)\ :\ t\in\partial\psi(s)\big\}.
\ee
\end{theo}

\begin{rema}
From \eqref{pcondopt} and the regularity results for elliptic equations, we deduce that the optimal control $V_{opt}$ is more regular if the function $h$ in \eqref{functionh} and the functions $j$ and $f$ satisfy some regularity assumptions. We recall that
\[\begin{cases}
h\text{ is continuous }\Longleftrightarrow \psi\text{ is strictly convex}\\
\ds h\text{ is Lipschitz continuous }\Longleftrightarrow \inf_{s_1,s_2\in \dom(\psi)\atop s_1<s_2}{d_-\psi(s_2)-d_+\psi(s_1)\over s_2-s_1}>0,
\end{cases}\]
where the functions $d_+\psi$ and $d_-\psi$ are defined by
$$\begin{cases}
\ds d_+\psi(s):=\lim_{\eps\searrow0}{\psi(s+\eps)-\psi(s)\over\eps}\in(-\infty,\infty]\\
\ds d_-\psi(s):=\lim_{\eps\searrow0}{\psi(s)-\psi(s-\eps)\over \eps}\in[-\infty,\infty)
\end{cases}$$
\end{rema}

The $BV$ regularity of the optimal control $V_{opt}$ can be deduced from the following result (see \cite{BCM24}).

\begin{theo}\label{Topr}
In addition to the conditions in Theorem \ref{expot} we assume $\O$ of class $C^{1,1}$, the function $g(t):=t\,h(t)$ non-decreasing in $t$, and for every $M>0$
$$\max_{|s|\le M}|\nabla_x\partial_s j(\cdot,s)|\in L^q(\O),\qquad \max_{|s|\le M}|\partial^2_{ss} j(\cdot,s)|\in L^1(\O),$$
with 
$$q\ge{2d\over d+1}\ \hbox{ if }1\le d\le2,\qquad q>{d\over2}\ \hbox{ if }d\ge3.$$
Then, for every $f\in BV(\O)\cap L^q(\O)$ we have
$$\bar u,v\in W^{2,q}(\O),\qquad \bar u vV_{opt}\in BV(\O),$$
with $\bar u$, $v$ the solutions of \eqref{pstate} and \eqref{padj} respectively.
\end{theo}

\begin{rema}
Under the assumptions of Theorem \ref{Topr}, the functions $\bar u$ and $v$ are continuous, and thus, the set $K:=\{\bar u v=0\}$ is a closed subset of $\overline\O$ which contains the boundary $\partial\O$. The fact that $\bar u vV_{opt}$ belongs to $BV(\O)$, proves then that $V_{opt}$ belongs to $BV_{loc}(\O\setminus K)$.
\end{rema}

\begin{rema}\label{bban}
We notice that, under the general assumptions we consider on the function $\psi$, higher regularity properties on the optimal potential $V_{opt}$ do not hold. For instance, when the function $\psi$ is of the form
$$\psi(s)=\begin{cases}
s&\text{if }s\in[\alpha,\beta]\qquad\text{(with $0\le\alpha<\beta$)}\\
+\infty&\text{otherwise,}
\end{cases}$$
the optimal control $V_{opt}$ is of {\it bang-bang} type, that is
$$\hat m=\alpha+(\beta-\alpha)1_E$$
for a suitable set $E$ which, by Theorem \ref{Topr}, is then a set with finite perimeter. In \cite{BCM24} one can find several numerical simulations that show the behavior of the set $E$ above in various situations.
\end{rema}

\begin{exam}\label{Ex31}
In order to show a numerical evidence of Remark \ref{bban}, we consider 
problem (\ref{poten})-(\ref{pstate}), with $\Omega\subset\mathbb{R}^2$ the unit ball,
$$j(x,s)=s,\qquad\psi(s)=\infty 1_{(-\infty,\alpha)\cup(\beta,+\infty)}+ks1_{[\alpha,\beta]},$$
and
$$f(x_1,x_2)=1_{\omega_1}(x_1,x_2)+1_{\omega_2}(x_1,x_2)+1_{\omega_3}(x_1,x_2).$$ 
where $\omega_i=B(c_i, r)$, $i=1,2$, are small balls of radius $r=0.2$ and centers $c_1=(0.35,0.45)$, $c_2=(-0.35,0.45)$ and $\omega_3$ the rectangle $[-0.5,0.5]\times[-0.5,-0.25]$, see Figure \ref{Fpo} feft. Taking $\alpha=0$, $\beta=1$ and $k=0.00225$, in Figure \ref{Fpo} right we show the computed bang-bang optimal potential.

\begin{figure}[!t]
\begin{minipage}[!h]{19cm}
\centering
\hspace{-3.5cm}\includegraphics[width=9.5cm]{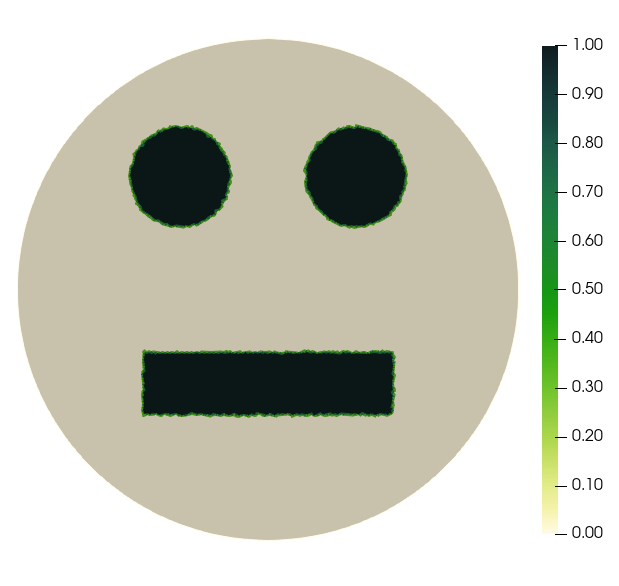}\hspace{-.2cm}
\includegraphics[width=9.5cm]{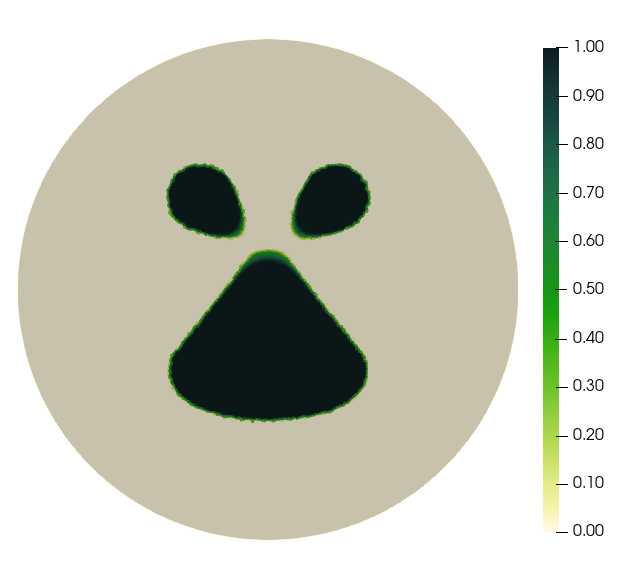}
\end{minipage}
\caption{Left: Right-hand side function $f$. Right: Computed optimal potential.} \label{Fpo}
\end{figure}
\end{exam}

\section{Optimal sources}\label{srhs}

In this section we consider optimal control problems for the Poisson equation on a bounded domain $\O\subset\R^d$ with homogeneous Dirichlet boundary conditions. The control variable is the source term $f$, constrained to an admissible class $\F$. The state equation is
\be\label{rhs}
\begin{cases}
-\Delta u=f&\hbox{in }\O\\
u\in H^1_0(\O).
\end{cases}\ee
with $u_f$ denoting the unique weak solution for a given $f$. The goal is to minimize a cost functional of the form
\be\label{J}
J(f)=\int_\O j(x,u_f,f)\,dx,
\ee
subject to the constraint $f\in\F$, where the admissible class $\F$ is defined by
$$\F=\bigg\{\int_\O\psi(f)\,dx\le m\bigg\},$$
for a convex, lower semicontinuous function $\psi:\R\to[0,\infty]$ such that
$$\begin{cases}
{\rm int}(D(\psi))\not=\emptyset\ \hbox{ with }D(\psi)=\big\{s\in\R:\ \psi(s)<\infty\}\\
\ds\lim_{|s|\to+\infty}\psi(s)=+\infty.
\end{cases}$$
The optimization problem then takes the form
\be\label{optpb}
\min\bigg\{\int_\O j(x,u_f,f)\,dx\ :\ \int_\O\psi(f)\,dx\le m\bigg\}.
\ee
A particularly interesting case occurs when $\psi(s)=+\infty$ outside an interval $[\alpha,\beta]$; in this setting, the optimal control may exhibit a {\it bang-bang} behavior, taking only the values $\alpha$ and $\beta$, that is
$$f=\beta1_E+\alpha1_{\O\setminus E}$$
for some measurable set $E\subset\O$. The problem then becomes a shape optimization problem, with $E$ as the control variable.

Our goal is to analyze the regularity of the optimal sources $f_{opt}$ which, in the case of bang-bang solutions, lead to the regularity of the optimal sets $E$ above. We refer to the article \cite{BCM26} for all the details, and for some numerical simulations that illustrate these phenomena and provide examples of optimal configurations.

A first existence result for optimal solutions is when the function $\psi$ satisfies the superlinear growth condition
\be\label{superl}
\lim_{|s|\to+\infty}\frac{\psi(s)}{|s|}=+\infty.
\ee
We also assume that the functional \eqref{J} is lower semicontinuous with respect to the weak $L^1(\O)$ topology, which is guaranteed (see for instance \cite {B89}) by assuming that the integrand $j(x, \cdot,\cdot)$ is lower semicontinuous in its arguments for almost every $x$, and that $j(x,s,\cdot)$ is convex for almost every $x$ and every $s$.

\begin{theo}\label{ex1}
Assume that the functional \eqref{J} is lower semicontinuous with respect to the weak $L^1(\O)$ topology, that the integrand $j(x,s,z)$ satisfies the growth condition
$$-c|s|^p-a(x)\le j(x,s,z),\qquad\text{with }c>0,\ a\in L^1(\O),\ p<d/(d-2),$$
and that the function $\psi$ satisfies the superlinear growth condition \eqref{superl}. Then the optimization problem \eqref{optpb} admits at least one solution $f_{opt}\in L^1(\O)$.
\end{theo}

When the function $\psi$ has only a linear growth:
\be\label{linear}
c|s|-a\le\psi(s)\qquad\text{for some constants }c>0,\ a\in\R,
\ee
the situation is more delicate, since optimal solutions can be measures, and the integral $\int_\O\psi(f)$ must be interpreted in the sense of measures, namely:
\be\label{psim}
\int_\O\psi(f)=\int_\O\psi\big(f^a(x)\big)\,dx+c^+(\psi)\int df_+^s-c^-(\psi)\int df_-^s
\ee
where $f=f^adx+f^s$ is the Radon-Nikodym decomposition of the measure $f$ into absolutely continuous and singular parts, $f_+$ and $f_-$ denote the positive and negative parts of $f$, and $c^-(\psi)$ and $c^+(\psi)$ are the recession limits of $\psi$, defined by
$$c^-(\psi)=\lim_{s\to-\infty}\frac{\psi(s)}{s}\qquad c^+(\psi)=\lim_{s\to+\infty}\frac{\psi(s)}{s}\;.$$
It is well-known (see \cite{B89}) that functionals of the form \eqref{psim} are lower semicontinuous with respect to the weak* convergence of measures.

The lower semicontinuity of the functional $J$ in \eqref{J} with respect to the weak* convergence of measures requires in this case more particular assumptions (see for instance \cite{BB90} for more general cases). Here we suppose the integrand $j(x,s,z)$ admits the decomposition in the form
$$j(x,s,z)=A(x,s)+B(x,z),$$
where the functions $A$ and $B$ satisfy the following properties:
\begin{itemize}
\item[-]the function $A(x,\cdot)$ is lower semicontinuous for almost every $x\in\O$;
\item[-]we have
$$A(x,s)\ge-c|s|^p+a(x)$$
for suitable $c>0$, $p<d/(d-2)$, $a\in L^1(\O)$;
\item[-]the function $B(x,\cdot)$ is convex and lower semicontinuous for almost every $x\in\O$;
\item[-]the recession function
$$B^\infty(x,z)=\lim_{t\to+\infty}\frac{B(x,tz)}{t}$$
is lower semicontinuous with respect to both variables $(x,z)$;
\item[-]we have
$$B(x,z)\ge a_0(x)z+a_1(x)$$
for suitable functions $a_0\in C_0(\O)$ and $a_1\in L^1(\O)$.
\end{itemize}
The assumptions above are sufficient to obtain the lower semicontinuity of the functional \eqref{J} with respect to the weak* convergence of measures, and we have the following existence result.

\begin{theo}\label{ex2}
Assume that the functional \eqref{J} verifies the conditions above for the lower semicontinuity with respect to the weak* convergence of measures, that the integrand $j$ satisfies the growth condition
$$-c|s|^p-a(x)\le j(x,s,z),\qquad\text{for some }c>0,\ a\in L^1(\O),\ p<d/(d-2),$$
and that the function $\psi$ has the linear growth \eqref{linear}. Then the optimization problem \eqref{optpb} admits at least one optimal solution $f_{opt}$, which is a measure with finite total variation.
\end{theo}

Once the issues related to the existence of solutions, as addressed in Theorems \ref{ex1} and \ref{ex2}, have been clarified, we are now in a position to examine the necessary optimality conditions that any solution to the problem must satisfy. To this end, it is convenient to introduce the resolvent operator $\RR$, which plays a central role in the subsequent analysis. The operator $\RR$ is defined by associating to each function $f$ the unique solution $u$ of the partial differential equation given by \eqref{rhs}. It is well-known that the operator $\RR$ is self-adjoint.

We begin our analysis with the case in which the function $\psi$ has a superlinear growth, as described in condition \eqref{superl}.

\begin{theo}\label{Thcop}
Let us assume that the integrand $j$, appearing in the definition of the cost functional in the optimal control problem \eqref{optpb}, satisfies the following growth condition:
$$|j(x,s,z)|\le a(x)+c|s|^p,\qquad\hbox{with }c>0,\ a\in L^1(\O),\ p<d/(d-2).$$
In addition, we suppose that the function $\psi$ satisfies the superlinear growth condition given in \eqref{superl}. Moreover we assume that, for almost every $x\in\O$ and for all $(s,z)\in\R^2$, the partial derivatives $\partial_sj(x,s,z)$ and $\partial_zj(x,s,z)$ exist and fulfill the inequalities:
$$\begin{cases}
|\partial_s j(x,s,z)|\le b(x)+\gamma\big(|s|^\sigma+|z|^\tau\big)\\
|\partial_z j(x,s,z)|\le\gamma,
\end{cases}$$
where $\gamma>0$, $b\in L^q(\O)$ with $q>d/2$, $\sigma<2/(d-2)$, and $\tau<2/d$.

Then, if $f_{opt}$ is an optimal solution to the problem \eqref{optpb}, there exists a non-negative scalar $\lambda\ge0$ such that, setting
$$w:=\RR\big(\partial_sj(x,\RR(f_{opt}),f_{opt})\big)+\partial_zj(x,\RR(f_{opt}),f_{opt}),$$
the following alternative holds:

\begin{itemize}
\item If $\lambda=0$, then
$$\begin{cases}
w\ge0\hbox{ a.e. in }\O\hbox{ if }\sup\big(\dom(\psi)\big)=+\infty\\
w\le0\hbox{ a.e. in }\O\hbox{ if }\inf\big(\dom(\psi)\big)=-\infty\\
f_{opt}=\min\big(\dom(\psi)\big)\hbox{ a.e. in }\big\{w>0\big\}\\
f_{opt}=\max\big(\dom(\psi)\big)\hbox{ a.e. in }\big\{w<0\big\}.
\end{cases}$$
\item If $\lambda>0$, then $\int_\O\psi(f_{opt})dx=m$ and
$$\psi\big(f_{opt}\big)+\psi^*\big(-{w\over\lambda}\big)=-{wf_{opt}\over\lambda}\qquad\hbox{ a.e. in }\O.$$
\end{itemize}
Moreover, if the function $j(x,\cdot,\cdot)$ is convex for almost every $x\in\O$, then the conditions stated above are not only necessary for optimality but also sufficient.
\end{theo}

When the function $\psi$ has a linear growth, as described in condition \eqref{linear} we have a similar result.

\begin{theo}\label{Thcop2}
Suppose that the function $\psi$ has a linear growth and that the integrand $j$ depends only on $(x,s)$ and not on $z$, with the growth condition
$$|j(x,s)|\le a(x)+c|s|^p,\qquad\hbox{with }c>0,\ a\in L^1(\O),\ p<d/(d-2).$$
We also assume that for almost every $x\in\O$ and every $s\in\R$, the partial derivative $\partial_sj(x,s)$ exists and satisfies
$$|\partial_s j(x,s)|\le b(x)+\gamma|s|^\sigma,$$
where again $\gamma>0$, $b\in L^q(\O)$ with $q>d/2$, $\sigma<2/(d-2)$.

Then, as above, if $f_{opt}$ is an optimal solution to the problem \eqref{optpb}, there exists a non-negative scalar $\lambda\ge0$ such that, setting
$$w:=\RR\big(\partial_sj(x,\RR(f_{opt}))\big)+\partial_zj(x,\RR(f_{opt})),$$
the following alternative holds:

\begin{itemize}
\item If $\lambda=0$, then
$$\begin{cases}
w\ge0\hbox{ a.e. in }\O\hbox{ if }\sup\big(\dom(\psi)\big)=+\infty\\
w\le0\hbox{ a.e. in }\O\hbox{ if }\inf\big(\dom(\psi)\big)=-\infty\\
f^a_{opt}=\min\big(\dom(\psi)\big)\hbox{ a.e. in }\big\{w>0\big\}\\
f^a_{opt}=\max\big(\dom(\psi)\big)\hbox{ a.e. in }\big\{w<0\big\}\\
{\rm supp}(f^s_{opt})\subset \{w=0\}.
\end{cases}$$
\item If $\lambda>0$, then $\int_\O\psi(f_{opt})=m$ and
$$\begin{cases}
\ds\psi\big(f_{opt}^a\big)+\psi^*\big(-{w\over\lambda}\big)=-{wf^a_{opt}\over\lambda}\hbox{ a.e. in }\O\\
-\lambda c^+(\psi)\le w\le -\lambda c^-(\psi)\hbox{ a.e. in }\O\\
{\rm supp}(f_{opt,+}^s)\subset\big\{w+\lambda c^+(\psi))=0\big\}\\
{\rm supp}(f_{opt,-}^s)\subset\big\{w+\lambda c^-(\psi))=0\big\}.
\end{cases}$$
\end{itemize}
Again, if the function $j(x,\cdot)$ is convex for almost every $x\in\O$, then the conditions stated above are not only necessary for optimality but also sufficient.
\end{theo}

\begin{exam}
We take
$$j(u)=-\frac{u^2}{2},\qquad\psi(s)=\frac{s^2}{2}.$$
The optimization problem is then
\be\label{eigpbf}
\max\bigg\{\frac{1}{2} \int_\O u_f^2\,dx\ :\ \int_\O f^2dx\leq 2m\bigg\}.
\ee
From Theorem \ref{Thcop} the necessary conditions of optimality give, for a suitable $\lambda>0$ and $w=-\RR\big(\RR(f)\big)$
$$w=-\lambda f\ \ \text{in }\O,\quad \int_\O f^2dx=2m.$$
Thus, $f$ solves the following eigenvalue problem for a fourth order PDE
$$\begin{cases}
\Delta^2f=f/\lambda&\text{in }\O\\
f=\Delta f=0&\text{on }\partial\O,
\end{cases}$$
or equivalently it is a solution of the eigenvalue problem for the Laplace operator
$$\begin{cases}
-\Delta f=f/\sqrt{\lambda}&\text{in }\O\\
f=0&\text{on }\partial\O.
\end{cases}$$
Since this implies $u_f=\sqrt{\lambda}f,$ we also have
$$\int_\O u_f^2dx=\lambda\int_\O f^2dx=2\lambda m.$$
Thus, $1/\sqrt{\lambda}$ agrees with the smaller eigenvalue of the Laplacian operator with homogeneous Dirichlet conditions. Denoting by $\mu_1$ such eigenvalue and taking $\phi$ as the unique positive eigenvector with unit $L^2$ norm, we have then proved that
$$\lambda=1/\mu_1^2,\quad f=\pm \sqrt{2m}\phi.$$

For $m=1/2$ and $\Omega=\big\{(x_1,x_2)\in\mathbb{R}^2: x_1^2+4x_2^2\le 4\big\}$ we show in Figure \ref{Eigv} the computed optimal source for problem \eqref{eigpbf} with the corresponding $\lambda=0.0785912$.

\begin{figure}[!t]
\includegraphics[width=12.5cm]{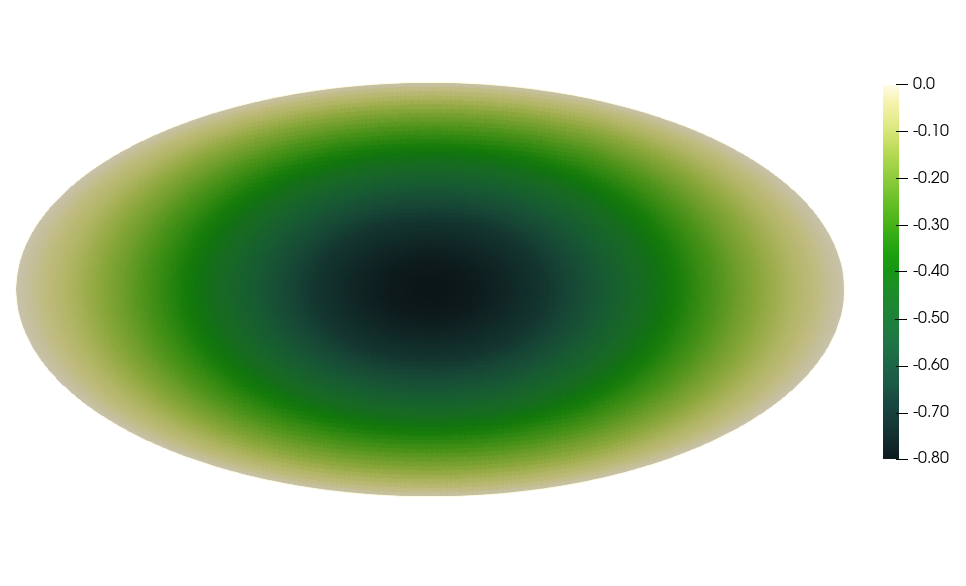}
\caption{Optimal source.} \label{Eigv}
\end{figure}
\end{exam}

Concerning the regularity of optimal sources $f_{opt}$, a general result is not available; we refer to \cite{BCM24} for a discussion on the matter. Here we limit ourselves to consider the particular case of the compliance functional with constraints on the source:
\be\label{complper}
\min\bigg\{\int_\O f\RR(f)\,dx\ :\ \int_\O f\,dx\ge m,\ f(x)\in[\alpha,\beta]\bigg\},
\ee
with $0\le\alpha<\beta$ and $\alpha|\O|<m<\beta|\O|$. By applying Theorem \ref{Thcop} we obtain that the optimal solution $f_{opt}$ is of bang-bang type, that is
\be\label{1E}
f_{opt}=\alpha1_E+\beta1_{\O\setminus E}\qquad\hbox{with }E=\{\RR(f_{opt})<s\},
\ee
for some positive constant $s$ that has to be chosen such that the integral constraint $\int_\O f\,dx\ge m$ is saturated. The function $u=\RR(f_{opt})$ thus solves te PDE
$$\begin{cases}
-\Delta u=\beta1_{\{u<s\}}+\alpha1_{\{u>s\}}&\text{in }\O\\
u=0&\text{on }\partial\O.
\end{cases}$$
We can now apply Theorem 3.5 of \cite{BCM24} and obtain the following result.

\begin{theo}\label{perimeter}
The optimal solution $f_{opt}$ of the minimization problem \eqref{complper} is in $BV(\O)$, hence the optimal set $E$ in \eqref{1E} above has a finite perimeter.
\end{theo}

When the domain $\O$ is convex, in some cases we can obtain a better regularity for the optimal right-hand side $f_{opt}$. Let us return to the compliance case \eqref{complper} with $\alpha=0$, and assume $\O$ convex. We have seen above that the optimal right-hand side $f_{opt}$ is in $BV(\O)$ and of bang-bang type: $f_{opt}=1_E$ with $E=\{w<s\}$ for a suitable $s$ such that $|E|=m$, where $w$ is the solution of the PDE
$$\begin{cases}
-\Delta w=1_{\{w<s\}}&\text{in }\O\\
w=0&\text{on }\partial\O.
\end{cases}$$
The result in \cite{BCM26} is the following.

\begin{theo}\label{Econvex}
The optimal set $E=\{w<s\}$ above is convex and of class $C^1$.
\end{theo}


\bigskip\bigskip

\noindent{\bf Acknowledgments.} The work of GB is part of the project 2017TEXA3H {\it``Gradient flows, Optimal Transport and Metric Measure Structures''} funded by the Italian Ministry of Research and University. GB is member of the Gruppo Nazionale per l'Analisi Matematica, la Probabilit\`a e le loro Applicazioni (GNAMPA) of the Istituto Nazionale di Alta Matematica (INdAM). The work of JCD and FM is a part of the FEDER project PID2023-149186NB-I00 of the {\it Ministerio de Ciencia, Innovaci\'on y Universidades} of the government of Spain.

\bigskip\noindent
Giuseppe Buttazzo:\\
Dipartimento di Matematica, Universit\`a di Pisa\\
Largo B. Pontecorvo 5, 56127 Pisa - ITALY\\
{\tt giuseppe.buttazzo@unipi.it}\\
{\tt http://www.dm.unipi.it/pages/buttazzo/}

\bigskip\noindent
Juan Casado-D\'{\i}az:\\
Dpto. de Ecuaciones Diferenciales y An\'alisis Num\'erico, Universidad de Sevilla\\
C. Tarfia s/n, 41012 Sevilla - SPAIN\\
{\tt jcasadod@us.es}\\
{\tt https://prisma.us.es/investigador/724}

\bigskip\noindent
Faustino Maestre:\\
Dpto. de Ecuaciones Diferenciales y An\'alisis Num\'erico, Universidad de Sevilla\\
C. Tarfia s/n, 41012 Sevilla - SPAIN\\
{\tt fmaestre@us.es}\\
{\tt https://prisma.us.es/investigador/2406}


\bigskip

\begin{thebibliography}{999}

\bibitem{All} {\sc G. Allaire: }{\it Shape optimization by the homogenization method.} Appl. Math. Sci. {\bf146}, Springer-Verlag, New York, 2002.

\bibitem{BB90}{G.~Bouchitt\'e, G.~Buttazzo: }{\it New lower semicontinuity results for nonconvex functionals defined on measures.} Nonlinear Anal., {\bf15} (1990), 679--692.

\bibitem{BB}{\sc G.~Bouchitt\'e, G.~Buttazzo: }{\it Characterization of optimal shapes and masses through Monge-Kantorovich equation.} J. Eur. Math. Soc., {\bf3} (2001), 139--168.

\bibitem{BBS}{\sc G.~Bouchitt\'e, G.~Buttazzo, P.~Seppecher: }{\it Shape optimization solutions via Monge-Kantorovich equation.} C. R. Acad. Sci. Paris, {\bf324-I} (1997), 1185--1191.

\bibitem{B89}{G.~Buttazzo: }{\it Semicontinuity, Relaxation and Integral Representation in the Calculus of Variations.} Pitman Res. Notes Math. Ser. {\bf207}, Longman, Harlow (1989).

\bib{BCM24}{G.~Buttazzo, J.~Casado-D\'{\i}az, F. Maestre}{On the regularity of optimal potentials in control problems governed by elliptic equations}{Adv. Calc. Var., {\bf17} (4) (2024), 1341--1364}

\bib{BCM25}{G.~Buttazzo, J.~Casado-D\'{\i}az, F.~Maestre}{Optimal coefficients for elliptic PDEs}{J. Convex Anal., (to appear), preprint available at {\tt https://arxiv.org} and at {\tt https://cvgmt.sns.it}}

\bib{BCM26}{G.~Buttazzo, J.~Casado-D\'{\i}az, F.~Maestre}{Optimal sources for elliptic PDEs}{Preprint available at {\tt https://arxiv.org} and at {\tt https://cvgmt.sns.it}}

\bibitem{BGL}{\sc G.~Buttazzo, M.S.~Gelli, D.~Lu\v ci\'c: }{\it Mass optimization problem with convex cost.} SIAM J. Math. Anal., {\bf55} (5) (2023), 5617--5642.

\bibitem{Cas} {\sc J.~Casado-D\'{\i}az:} {\it Some smoothness results for the optimal design of a two composite material which minimizes the energy.} Calc. Var. Partial Differential Equations, {\bf53} (2015), 649--673.

\bibitem{Cas2} {\sc J.~Casado-D\'{\i}az:} {\it Smoothness properties for the optimal mixture of two isotropic materials. The compliance and eigenvalue problems.} SIAM J. Cont. Optim., {\bf53} (2015), 2319--2349.

\bibitem{Cab}{\sc E.~Cabib: }{\it A relaxed control problem for two-phase conductors.} Ann. Univ. Ferrara, {\bf33} (1987), 207--218.

\bibitem{CDM}{\sc E.~Cabib, G.~Dal Maso: }{\it On a class of optimum problems in structural design.} J. Optimization Theory Appl., {\bf56} (1988), 39--65.

\bibitem{Cl}{\sc F.H. Clarke: }{\it Multiple integrals of Lipschitz functions in the calculus of variations.} Proc. Amer. Math. Soc., {\bf64} (1977), 260--264.

\bibitem{DGS}{\sc E.~De~Giorgi, S.~Spagnolo: }{\it Sulla convergenza degli integrali dell'energia per operatori ellittici del secondo ordine.} Boll. Un. Mat. Ital., {\bf8} (4) (1973), 391--411.

\bibitem{Ek}{\sc I.~Ekeland: }{\it Th\'eorie des jeux.} Presses Univ. France, Paris, 1974.

\bibitem{LC}{\sc K.A.~Lurie, A.V.~Cherkaev: }{\it Exact estimates of the conductivity of a binary mixture of
isotropic materials.} Proc. Royal Soc. Edinburgh, {\bf104 A} (1986), 21--38.

\bibitem{Mur}{\sc F.~Murat: }{\it Contre-exemples pour divers probl\`emes ou le contr\^ole intervient dans les coefficients.} Ann. Mat. Pura Appl., {\bf112} (1977), 49--68.

\bibitem{MT}{\sc F.~Murat, L.~Tartar: }{\it Optimality conditions and homogenization.} Proceedings of ``Nonlinear variational problems'', Res. Notes in Math. {\bf127}, Pitman, London, (1985), 1--8.

\bibitem{Tar}{\sc L.~Tartar: }{\it Estimations fines des coefficients homog\'en\'eises.} Proceedings of ``Ennio De Giorgi Colloquium'', Res. Notes in Math. {\bf125}, Pitman, London, (1985), 168--187.

\end{thebibliography}
\end{document}